\documentclass[letterpaper,10pt]{article}

\usepackage[margin=1in]{geometry}

\usepackage{amsthm,amssymb,amsmath,graphicx}
\allowdisplaybreaks
\usepackage{textcomp}
\usepackage{algorithm}
\usepackage{algpseudocode}
\makeatletter
\newenvironment{breakablealgorithm}
  {
   \begin{center}
     \refstepcounter{algorithm}
     \hrule height.8pt depth0pt \kern2pt
     \renewcommand{\caption}[2][\relax]{
       {\raggedright\textbf{\ALG@name~\thealgorithm} ##2\par}%
       \ifx\relax##1\relax 
         \addcontentsline{loa}{algorithm}{\protect\numberline{\thealgorithm}##2}%
       \else 
         \addcontentsline{loa}{algorithm}{\protect\numberline{\thealgorithm}##1}%
       \fi
       \kern2pt\hrule\kern2pt
     }
  }{
     \kern2pt\hrule\relax
   \end{center}
  }
\makeatother
\usepackage{chngcntr}
\counterwithin{figure}{section}
\counterwithin{table}{section}
\counterwithin{equation}{section}
\counterwithin{algorithm}{section}
\usepackage[labelfont=bf]{caption}
\usepackage[nottoc,notlot,notlof]{tocbibind}
\usepackage{mathtools}
\usepackage{csquotes}

\theoremstyle{definition}
\newtheorem{theorem}{Theorem}[section]

\usepackage[usenames, dvipsnames]{color}
\usepackage{hyperref}
\hypersetup{
  colorlinks=true,
  urlcolor=blue,
  linkcolor=MidnightBlue,
  citecolor=ForestGreen,
  pdfinfo={
    CreationDate={D:20180508114800},
    ModDate={D:20180508114800},
  },
}

\usepackage{embedfile}
\embedfile{\jobname.tex}

\usepackage{fancyhdr}
\newcommand{\cond}[1]{\operatorname{cond}\left(#1\right)}
\newcommand{\fl}[1]{\operatorname{fl}\left(#1\right)}
\newcommand{\bigO}[1]{\mathcal{O}\left(#1\right)}
\newcommand{\mach}{\mathbf{u}}

\begin{document}

\begin{abstract}
\noindent  In computer aided geometric design a polynomial is usually
represented in Bernstein form. The de Casteljau algorithm is the most
well-known algorithm for evaluating a polynomial in this form. Evaluation
via the de Casteljau algorithm has relative forward error proportional to
the condition number of evaluation. However, for a particular family of
polynomials, a curious phenomenon occurs: the observed error is much smaller
than the expected error bound. We examine this family and prove a much
stronger error bound than the one that applies to the general case. Then
we provide a few examples to demonstrate the difference in rounding.
\end{abstract}

\tableofcontents

\section{Introduction}

In computer aided geometric design, polynomials are usually expressed in
Bernstein form. Polynomials in this form are usually evaluated by the
de Casteljau algorithm. This algorithm has a round-off error bound
which grows only linearly with degree, even though the number of
arithmetic operations grows quadratically. The Bernstein basis is
optimally suited (\cite{Farouki1987, Delgado2015, Mainar2005})
for polynomial evaluation. Nevertheless the de Casteljau
algorithm returns results arbitrarily less accurate than the working
precision \(\mach\) when evaluating \(p(s)\) is ill-conditioned.
The relative accuracy of the computed
evaluation with the de Casteljau algorithm (\texttt{DeCasteljau}) satisfies
(\cite{Mainar1999}) the following a priori bound:
\begin{equation}\label{de-casteljau-error}
  \frac{\left|p(s) - \mathtt{DeCasteljau}(p, s)\right|}{\left|p(s)\right|} \leq
  \cond{p, s} \times \bigO{\mach}.
\end{equation}

\begin{figure}
  \includegraphics[width=0.75\textwidth]{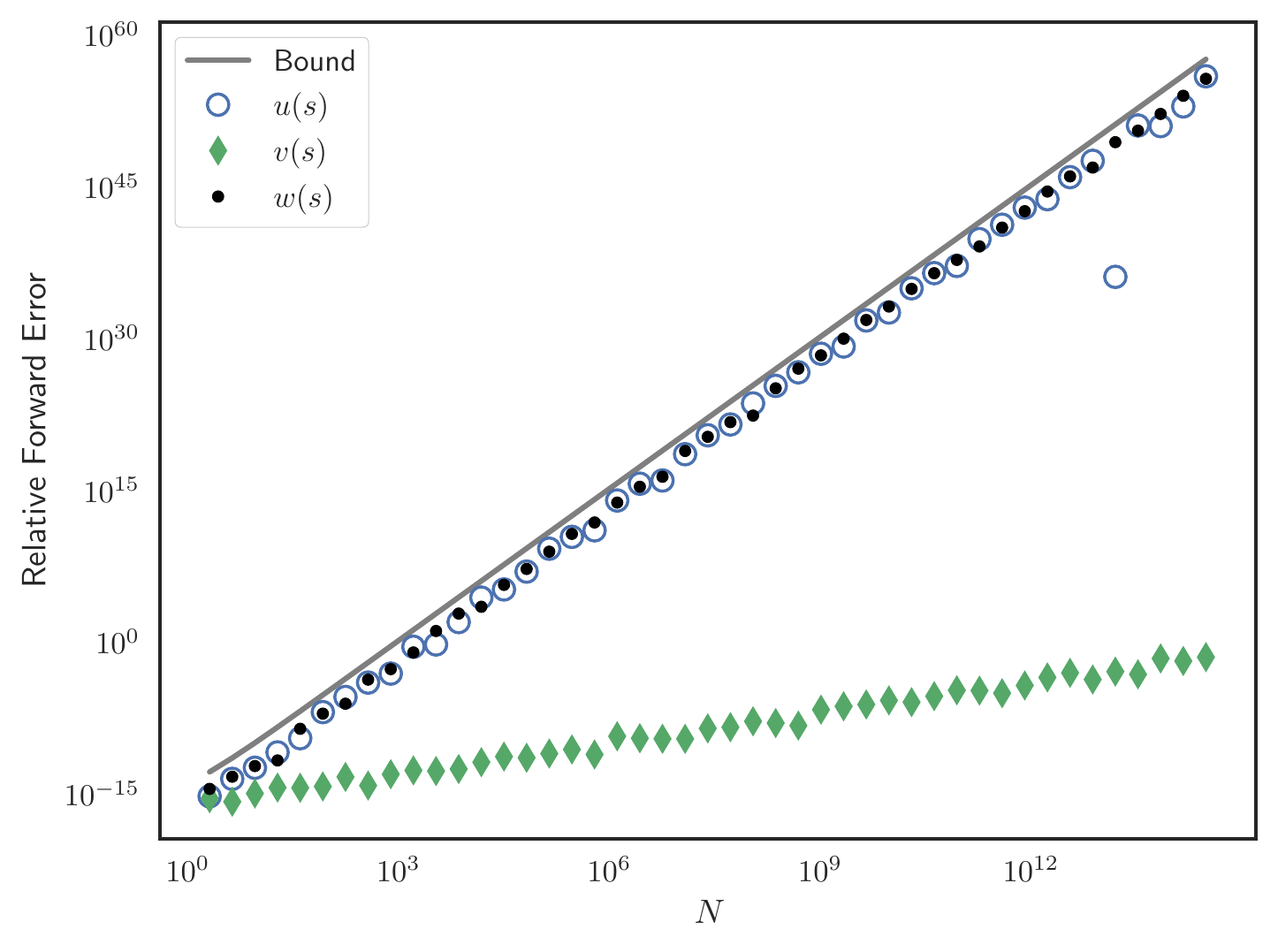}
  \centering
  \captionsetup{width=.75\linewidth}
  \caption{Comparing relative forward error to a priori bound for
    \(u(s) = (1 - 4s)^5, v(s) = (1 - 5s)^5\) and \(w(s) = (1 - 6s)^5\)}
  \label{fig:against-a-priori}
\end{figure}

For example, consider \(u(s) = (1 - 4s)^5, v(s) = (1 - 5s)^5\) and
\(w(s) = (1 - 6s)^5\). For points of the form
\begin{equation}
s_u = \frac{1}{4} + \frac{6}{16N}, \quad
s_v = \frac{1}{5} + \frac{8}{25N}, \quad
s_w = \frac{1}{6} + \frac{10}{36N}
\end{equation}
which are near the multiple roots, the condition numbers of evaluation are:
\begin{equation}
\cond{u, s_u} = \cond{v, s_v} = \cond{w, s_w} =
  \left|N\right|^5 + \bigO{N^4}.
\end{equation}
As we can see in Figure~\ref{fig:against-a-priori}, one of these is not
like the others. Evaluation of both \(u(s)\) and \(w(s)\) produces
relative forward error very close to the a prior bound\footnote{There
are actually three different error bounds, but the \(\left|N\right|^5\)
term dominates so much that they are not visually discernible.}. However,
the observed relative error when evaluating \(v(s)\) is significantly
lower than expected.

As we'll explore in Section~\ref{sec:improved-bound}, \(v(s)\) belongs
to a family of polynomials that the de Casteljau method can evaluate with
a significantly higher level of accuracy than expected. Notice that
\begin{equation}
v(s) = (1 - 5s)^5 = \left[(1 - s) - 4s\right]^5 =
  B_{0, 5}(s) - 4 B_{1, 5}(s) +
  16 B_{2, 5}(s) + \cdots;
\end{equation}
in particular, the Bernstein coefficients are powers of \(2\) (up to sign).
The aforementioned family of polynomials contains any of the form
\begin{equation}
b_0 \left[(1 - s) - 2^t s\right]^n =
  b_0 B_{0, n}(s) - \left(b_0 2^t\right) B_{1, n}(s) +
  \left(b_0 2^{2t}\right) B_{2, n}(s) + \cdots.
\end{equation}
Polynomials in this family have coefficients that can be represented
exactly (i.e. with no round-off).

The paper is organized as follows. Section~\ref{sec:notation} establishes
notation for error analysis with floating point operations and reviews the
de Casteljau algorithm. In
Section~\ref{sec:improved-bound}, the lowered error bound is proved
for polynomials in the special family and numerical experiments compare
observed relative error to the newly improved bound.
Finally, in Section~\ref{sec:implications} we comment on the impact
that this lowered bound makes on comparisons between the de Casteljau
algorithm and the VS algorithm.

\section{Basic notation and results}\label{sec:notation}

\subsection{Floating Point and Forward Error Analysis}

We assume all floating point operations obey
\begin{equation}
  a \star b = \fl{a \circ b} = (a \circ b)(1 + \delta_1) =
  (a \circ b) / (1 + \delta_2)
\end{equation}
where \(\star \in \left\{\oplus, \ominus, \otimes, \oslash\right\}\), \(\circ
\in \left\{+, -, \times, \div\right\}\) and \(\left|\delta_1\right|,
\left|\delta_2\right| \leq \mach\). The symbol \(\mach\) is the unit round-off
and \(\star\) is a floating point operation, e.g.
\(a \oplus b = \fl{a + b}\). (For IEEE-754 floating point double precision,
\(\mach = 2^{-53}\).) We denote the computed result of
\(\alpha \in \mathbf{R}\) in floating point arithmetic by
\(\widehat{\alpha}\) or \(\fl{\alpha}\) and use \(\mathbf{F}\) as the set of
all floating point numbers (see \cite{Higham2002} for more details).
Following \cite{Higham2002}, we will use the following classic properties in
error analysis.

\begin{enumerate}
  \item If \(\delta_i \leq \mach\), \(\rho_i = \pm 1\), then
      \(\prod_{i = 1}^n (1 + \delta_i)^{\rho_i} = 1 + \theta_n\),
  \item \(\left|\theta_n\right| \leq \gamma_n \coloneqq
      n \mach / (1 - n \mach)\),
  \item \((1 + \theta_k)(1 + \theta_j) = 1 + \theta_{k + j}\),
  \item \(\gamma_k + \gamma_j + \gamma_k \gamma_j \leq \gamma_{k + j}
    \Longleftrightarrow (1 + \gamma_k)(1 + \gamma_j) \leq 1 + \gamma_{k + j}\),
  \item \((1 + \mach)^j \leq 1 / (1 - j \mach) \Longleftrightarrow
  (1 + \mach)^j - 1 \leq \gamma_j\).
\end{enumerate}

\begin{theorem}\label{thm:floating-point-exact}
In the absence of overflow or underflow, for \(a, b, -2^t \in \mathbf{F}\)
\begin{align}
\left(-2^t a\right) \otimes b &= -2^t \left(a \otimes b\right) \\
\left(-2^t a\right) \oplus \left(-2^t b\right) &= -2^t \left(a \oplus b\right).
\end{align}
\end{theorem}

\subsection{Bernstein Basis and de Casteljau Algorithm}

A polynomial written in the Bernstein basis is of the form
\begin{equation}
p(s) = \sum_{j = 0}^n b_j B_{j, n}(s)
\end{equation}
where \(B_{j, n}(s) = \binom{n}{j} (1 - s)^{n - j} s^j\). When
\(s \in \left[0, 1\right]\), the Bernstein basis functions are
non-negative. We refer to \(\widetilde{b}_j \coloneqq \binom{n}{j} b_j\)
as the \textit{scaled Bernstein coefficients}.
The condition number of evaluation for \(p(s)\) is
\begin{equation}
\cond{p, s} = \frac{\widetilde{p}(s)}{\left|p(s)\right|}
\end{equation}
where
\(\widetilde{p}(s) \coloneqq \sum_{j = 0}^n \left|b_j\right| B_{j, n}(s)\).

\begin{breakablealgorithm}
  \caption{\textit{de Casteljau algorithm for polynomial evaluation.}}
  \label{alg:de-casteljau}

  \begin{algorithmic}
    \Function{\(\mathtt{result} = \mathtt{DeCasteljau}\)}{$b, s$}
      \State \(n = \texttt{length}(b) - 1\)
      \State \(\widehat{r} = 1 \ominus s\)
      \\
      \For{\(j = 0, \ldots, n\)}
        \State \(\widehat{b}_j^{(n)} = b_j\)
      \EndFor
      \\
      \For{\(k = n - 1, \ldots, 0\)}
        \For{\(j = 0, \ldots, k\)}
          \State \(\widehat{b}_j^{(k)} = \left(
              \widehat{r} \otimes \widehat{b}_j^{(k + 1)}\right) \oplus
              \left(s \otimes \widehat{b}_{j + 1}^{(k + 1)}\right)\)
        \EndFor
      \EndFor
      \\
      \State \(\mathtt{result} = \widehat{b}_0^{(0)}\)
    \EndFunction
  \end{algorithmic}
\end{breakablealgorithm}

\begin{theorem}\label{thm:de-casteljau-bound}
The de Casteljau algorithm (Algorithm~\ref{alg:de-casteljau}) satisfies
\begin{equation}
\left|p(s) - \mathtt{DeCasteljau}(p, s)\right| \leq \gamma_{3n}
  \widetilde{p}(s).
\end{equation}
\end{theorem}

\begin{proof}
See Appendix~\ref{sec:appendix-proof-details}.
\end{proof}

\section{Improved Bound}\label{sec:improved-bound}

We seek to analyze our family of polynomials of the form:
\begin{equation}
p(s) = b_0 \left[(1 - s) - 2^t s\right]^n.
\end{equation}
When written in the Bernstein basis, \(p(s)\) has coefficients that satisfy
\begin{equation}\label{exact-ratio}
b_{j + 1} = -2^t b_j.
\end{equation}
In a finite precision binary arithmetic, \eqref{exact-ratio} will hold
exactly (i.e. with no round-off) until overflow or underflow makes it
impossible to represent
\(b_j\) in the given arithmetic (the mantissa will always be the same but
the sign and exponent will change). This useful property remains
true for the intermediate terms computed by the de Casteljau method:
\begin{equation}
b_{j + 1}^{(k)} = (1 - s) b_{j + 1}^{(k + 1)} + s b_{j + 2}^{(k + 1)} =
  -2^t\left[(1 - s) b_j^{(k + 1)} + s b_{j + 1}^{(k + 1)}\right] =
  -2^t b_j^{(k)}.
\end{equation}
Remarkably, this also holds true for the \textbf{computed} values:
\(\widehat{b}_{j + 1}^{(k)} =
-2^t \widehat{b}_j^{(k)}\). Following Theorem~\ref{thm:floating-point-exact}
we have
\begin{align}
\widehat{b}_{j + 1}^{(k)} &= \left[\widehat{r} \otimes
  \widehat{b}_{j + 1}^{(k + 1)}\right] \oplus
  \left[s \otimes \widehat{b}_{j + 2}^{(k + 1)}\right] \\
&= \left[-2^t \left(\widehat{r} \otimes
  \widehat{b}_j^{(k + 1)}\right)\right] \oplus
  \left[-2^t\left(s \otimes \widehat{b}_{j + 1}^{(k + 1)}\right)\right] \\
&= -2^t\left(\left[\widehat{r} \otimes
  \widehat{b}_j^{(k + 1)}\right] \oplus
  \left[s \otimes \widehat{b}_{j + 1}^{(k + 1)}\right]\right) \\
&= -2^t \widehat{b}_j^{(k)}.
\end{align}
Thus, for such \(p(s)\), we only need compute
\(\widehat{b}_0^{(k)}\):
\begin{equation}
\widehat{b}_0^{(k)} = \left[\widehat{r} \otimes
  \widehat{b}_0^{(k + 1)}\right] \oplus \left[s \otimes
  \widehat{b}_1^{(k + 1)}\right] = \left[\widehat{r} \otimes
  \widehat{b}_0^{(k + 1)}\right] \oplus \left[\left(-2^t\right) \left(s \otimes
  \widehat{b}_0^{(k + 1)}\right)\right].
\end{equation}

\begin{theorem}\label{thm:better-error}
For a polynomial of the form
\begin{equation}
p(s) = b_0 \left[(1 - s) - 2^t s\right]^n
\end{equation}
the relative accuracy of the computed
evaluation with the de Casteljau algorithm satisfies
the following a priori bound:
\begin{equation}\label{better-de-casteljau-error}
\frac{\left|p(s) - \mathtt{DeCasteljau}(p, s)\right|}{\left|p(s)\right|}
  \leq \left(1 + \left|\phi\right| \gamma_3\right)^n - 1
\end{equation}
where
\begin{equation}
\phi \coloneqq \frac{(1 - s) + 2^t s}{(1 - s) - 2^t s} =
  \frac{1 + \left(2^t - 1\right) s}{1 - \left(2^t + 1\right) s}.
\end{equation}
\end{theorem}

\begin{proof}
Let \(r = 1 - s\) and \(\widehat{r} = 1 \ominus s\).
For \(k < n\),
\begin{equation}
\widehat{b}_0^{(k)} = \left[\widehat{r} \otimes
  \widehat{b}_0^{(k + 1)}\right] \oplus \left[\left(-2^t\right)
  \left(s \otimes \widehat{b}_0^{(k + 1)}\right)\right] =
  \widehat{b}_0^{(k + 1)} \left[r(1 + \theta_3) - 2^t s(1 + \theta_2)\right].
\end{equation}
Note that this can be written as \(\widehat{b}_0^{(k + 1)}
\left[\left(r - 2^t s\right) + E\right]\) where the round-off term
satisfies \(\left|E\right| \leq \left(r + 2^t s\right) \gamma_3\) for
\(s \in \left[0, 1\right]\). Hence we write
\(\widehat{b}_0^{(k)} = \widehat{b}_0^{(k + 1)} \left[\left(r - 2^t s\right) +
  \left(r + 2^t s\right)\theta_3\right]\).
Since \(\widehat{b}_0^{(n)} = b_0\), we have
\begin{equation}
\widehat{b}_0^{(0)} = b_0 \left[\left(r - 2^t s\right) +
  \left(r + 2^t s\right)\theta_3\right]^n.
\end{equation}
Dividing this by \(b_0^{(0)} = b_0 \left(r - 2^t s\right)^n\) we have
\begin{equation}
\widehat{b}_0^{(0)} = b_0^{(0)} \left(1 + \phi \cdot \theta_3\right)^n.
\end{equation}
Hence our relative error can be bound by
\begin{equation}
\left|\left(1 + \phi \cdot \theta_3\right)^n - 1\right| \leq
\left(1 + \left|\phi\right| \gamma_3\right)^n - 1
\end{equation}
as desired.
\end{proof}

Since
\begin{equation}
\widetilde{p}(s) = \sum_{j = 0}^n \left|b_0 \left(-2^t\right)^j\right|
  B_{j, n}(s) = \left|b_0\right| \left[(1 - s) + 2^t s\right]^n
\end{equation}
we see that \(\cond{p, s} = \widetilde{p}(s) / \left|p(s)\right| =
\left|\phi\right|^n\). So we can compare our improved bound
\(\left(1 + \left|\phi\right| \gamma_3\right)^n - 1\)
to the na\"ive bound \(\gamma_{3n} \left|\phi\right|^n\).

\begin{figure}
  \includegraphics[width=0.75\textwidth]{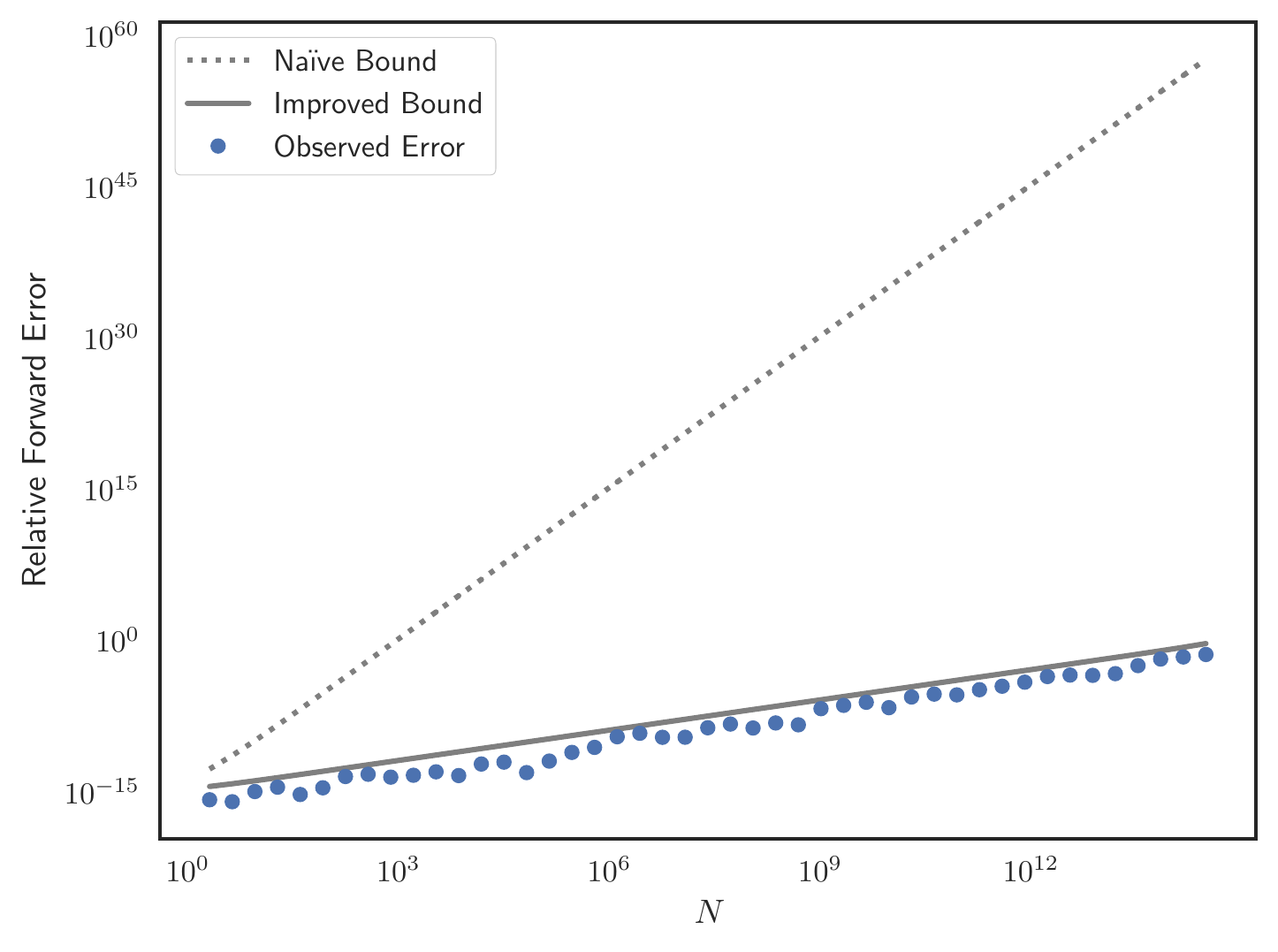}
  \centering
  \captionsetup{width=.75\linewidth}
  \caption{Comparing na\"ive relative error bound to improved bound for
    \(p(s) = (1 - 5s)^{5}\) evaluated at \(s = 1/5 + 8/(25N)\)}
  \label{fig:curbed-condition}
\end{figure}

For example, we'll use \(b_0 = 1, n = 5, t = 2\), i.e. \(p(s) = (1 - 5s)^5\)
for a numerical experiment with. We compare the na\"ive and improved bounds to
the observed relative forward error in Figure~\ref{fig:curbed-condition}.
The figure shows the evaluation of \(p(s)\) at the points
\(\left\{\frac{1}{5} + \frac{8}{25 \cdot 2.1^e}\right\}_{e = 1}^{45}\),
which cause the condition number of evaluation to grow exponentially.
The form \(s = 1/5 + 8/(25N)\) is chosen because it simplifies
\(\phi\):
\begin{equation}
\phi = \frac{1 + 3s}{1 - 5s} = -N - \frac{3}{5}.
\end{equation}
As can be seen in the figure, the observed errors closely match the
improved bound and are significantly smaller than
the na\"ive bound. What's more, the actual error is still useful
(i.e. less than \(\bigO{1}\)) when the condition number becomes
very large.

\section{Implications}\label{sec:implications}

In \cite{Delgado2015}, the de Casteljau algorithm is compared to the VS
algorithm (\cite{Schumaker1986}) along with a few other methods. The VS
algorithm relies on two transformations of \(p(s) =
\sum_{j = 0}^n \widetilde{b}_j (1 - s)^{n - j} s^j\) using
\(\sigma_1 = (1 - s) / s\) and \(\sigma_2 = s / (1 - s)\):
\begin{equation}\label{vs-transform}
p(s) = s^n \left[\widetilde{b}_0 \sigma_1^n + \cdots + \widetilde{b}_n\right]
= (1 - s)^n \left[\widetilde{b}_0 + \cdots + \widetilde{b}_n \sigma_2^n\right].
\end{equation}
See Appendix~\ref{sec:vs-algorithm} for more details on the VS algorithm,
in particular Algorithm~\ref{alg:vs-algorithm} which describes the method and
Theorem~\ref{thm:vs-bound} which provides an error bound.

\begin{figure}
  \includegraphics[width=0.9375\textwidth]{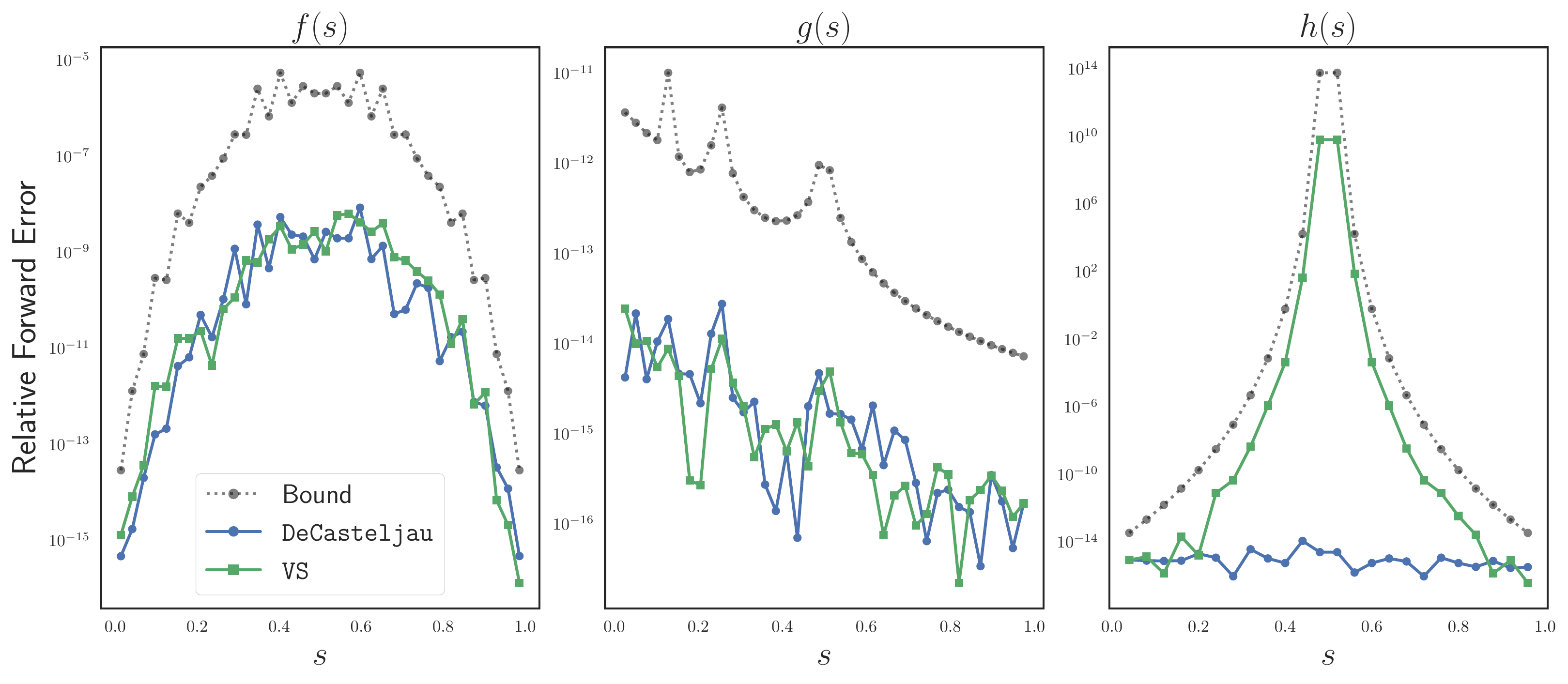}
  \centering
  \captionsetup{width=.75\linewidth}
  \caption{Comparing the de Casteljau and VS algorithms.}
  \label{fig:dp15-compare}
\end{figure}

In \cite{Delgado2015}, the authors use\footnote{The coefficients of \(f(s)\)
and \(g(s)\) cannot be represented exactly in IEEE-754 floating point
double precision.}
\begin{equation}
f(s) = \prod_{j = 1}^{20} (s - j/20), \quad
g(s) = \prod_{j = 1}^{20} (s - 2/2^j), \quad
h(s) = (s - 1/2)^{20}
\end{equation}
to compare the relative error of various methods.
Figure~\ref{fig:dp15-compare} reproduces the numerical experiments they
performed by evaluating \(f(s)\) at
\(\left\{\frac{2j - 1}{72}\right\}_{j = 1}^{36}\),
\(g(s)\) at \(\left\{\frac{j}{39}\right\}_{j = 1}^{38}\) and
\(h(s)\) at \(\left\{\frac{4j}{100}\right\}_{j = 1}^{24}\).
In addition, we have included a dotted line marking the
a prior error bound for each point of evaluation.
Viewing the errors for \(f(s)\) and \(g(s)\), there is no qualitative
difference between the de Casteljau algorithm and the VS algorithm, so
the authors use \(h(s)\) as the tiebreaker to conclude
\begin{displayquote}
the algorithm with a good behavior everywhere is the de Casteljau algorithm.
\end{displayquote}
However,
\begin{equation}
h(s) = \left(s - 1/2\right)^{20} = 2^{-20} (2s - 1)^{20}
  = 2^{-20} \left[(1 - s) - s\right]^{20},
\end{equation}
i.e. it is a member of our special family of polynomials with
\(b_0 = 2^{-20}, n = 20, t = 0\). This gives the de Casteljau
algorithm an unfair advantage over the VS algorithm in a
test case that is not representative of general polynomials.

Given the quadratic growth in the number of arithmetic operations
performed in the de Casteljau algorithm, an alternative with
linear growth (i.e. the VS method) should not be discounted if
it produces similar results. The special family of polynomials
explored here explains why the de Casteljau algorithm performs
so much better than the VS method and shows that such polynomials
are not indicative of the accuracy of the two methods
relative to one another.

\section{Acknowledgements}

The author would like to thank W. Kahan the publication
(\cite{Kahan1972}) that motivated the title of this work and
is closely related to the topic at hand.

\bibliography{curious-case}
\bibliographystyle{alpha}

\appendix

\section{Proof Details}\label{sec:appendix-proof-details}

\begin{proof}[Proof of Theorem~\ref{thm:de-casteljau-bound}]
When using the de Casteljau method, we have
\(b_j^{(n)} = \widehat{b}_j^{(n)} = b_j\) and for \(k = n - 1, \ldots, 0\)
and \(j = 0, \ldots, k\):
\begin{align}
b_j^{(k)} &= (1 - s) b_j^{(k + 1)} + s b_{j + 1}^{(k + 1)} \\
\widehat{b}_j^{(k)} &= \left[(1 \ominus s) \otimes
  \widehat{b}_j^{(k + 1)}\right]
  \oplus \left[s \otimes \widehat{b}_{j + 1}^{(k + 1)}\right].
\end{align}
This means that
\begin{equation}
\widehat{b}_j^{(k)} = (1 - s) \widehat{b}_j^{(k + 1)} (1 + \theta_3)
+ s \widehat{b}_{j + 1}^{(k + 1)} (1 + \theta_2)
\end{equation}
so that
\begin{align}
\widehat{b}_0^{(0)} &= (1 - s) \widehat{b}_0^{(1)} (1 + \theta_3)
+ s \widehat{b}_{1}^{(1)} (1 + \theta_2) \\
&= \widehat{b}_0^{(2)} B_{0, 2}(s) (1 + \theta_6) +
\widehat{b}_1^{(2)} B_{1, 2}(s) (1 + \theta_5) +
\widehat{b}_2^{(2)} B_{2, 2}(s) (1 + \theta_4) \\
&= \sum_{j = 0}^n \widehat{b}_j^{(n)} B_{j, n}(s) (1 + \theta_{3n - j}) \\
&= b_0^{(0)} + \sum_{j = 0}^n b_j B_{j, n}(s) \theta_{3n - j}.
\end{align}
Hence we have
\begin{equation}
\left|p(s) - \mathtt{DeCasteljau}(p, s)\right| \leq \gamma_{3n}
  \widetilde{p}(s)
\end{equation}
as desired. Note that this differs from the bound given in
\cite{Mainar1999}, Corollary 3.2 because the authors don't consider the
round-off when computing \(\widehat{r}\).
\end{proof}

\section{VS Algorithm}\label{sec:vs-algorithm}

In \cite{Schumaker1986}, a modified form of Horner's method is
described for evaluating a polynomial in Bernstein form.
Following~\eqref{vs-transform}, the algorithm applies Horner's
method to the scaled Bernstein coefficients with an input
related to \(s\):

\begin{breakablealgorithm}
  \caption{\textit{VS algorithm for polynomial evaluation.}}
  \label{alg:vs-algorithm}

  \begin{algorithmic}
    \Function{\(\mathtt{result} = \mathtt{VS}\)}{$b, s$}
      \State \(n = \texttt{length}(b) - 1\)
      \State \(\widehat{r} = 1 \ominus s\)
      \If{\(s \geq 1/2\)}
        \State \(\widehat{\sigma} = \widehat{r} \oslash s\)
        \State \(m = s\)
        \State \(\left[c_0, \ldots, c_n\right] =
            \left[b_n, \ldots, b_0\right]\)
      \Else
        \State \(\widehat{\sigma} = s \oslash \widehat{r}\)
        \State \(m = \widehat{r}\)
        \State \(\left[c_0, \ldots, c_n\right] =
            \left[b_0, \ldots, b_n\right]\)
      \EndIf
      \\
      \State \(\widehat{p}_n = c_n\)
      \For{\(k = n - 1, \ldots, 0\)}
        \State \(\widehat{p}_k = \left[\widehat{\sigma} \otimes
            \widehat{p}_{k + 1}\right] + \left[
            \binom{n}{k} \otimes c_{k}\right]\)
      \EndFor
      \\
      \State \(\widehat{m}_1 = \widehat{m}\)
      \For{\(k = 2, \ldots, n\)}
        \State \(\widehat{m}_k = \widehat{m}_{k - 1} \otimes \widehat{m}\)
      \EndFor
      \\
      \State \(\mathtt{result} = \widehat{m}_n \otimes \widehat{p}_0\)
    \EndFunction
  \end{algorithmic}
\end{breakablealgorithm}

This has the benefit of using a linear number of floating point operations,
as compared to the de Casteljau method, which uses a quadratic
number of floating point operations.

\begin{theorem}\label{thm:vs-bound}
The value computed by the VS algorithm
(Algorithm~\ref{alg:vs-algorithm}) satisifies\footnote{
The coefficients \(\gamma_{5n}\) and \(\gamma_{6n}\) differ from
\(\gamma_{4n}\) in \cite{Delgado2009} (Theorem 4.2) because the authors
don't account for the multiplication by \(\binom{n}{k}\) in computing the
scaled Bernstein coefficients or the round-off in \(1 \ominus s\) when
computing \((1 - s)^n\).}
\begin{equation}
\left|p(s) - \mathtt{VS}(p, s)\right| \leq
  \begin{cases}
    \gamma_{6n} \widetilde{p}(s) \quad
      \text{when } s \in \left[0, 1/2\right) \\
    \gamma_{5n} \widetilde{p}(s) \quad \text{when } s \in \left[1/2, 1\right].
  \end{cases}
\end{equation}
\end{theorem}

\begin{proof}
Since the algorithm has two branches depending on \(s \geq 1/2\), we
have to make a few distinctions throughout. However, most arguments
apply to both branches of the algorithm. The following analysis
\textbf{assumes}
that there is no round-off introduced by the computation of binomial
coefficients. For \(n \leq 56\), \(\binom{n}{k}\) can be represented exactly
in IEEE-754 floating point double precision for all \(k\) but
\(\binom{57}{25}\) is the
``first'' binomial coefficient that must be rounded.

In either case, computing
\(\widehat{\sigma} = (1 \ominus s) \oslash s\) or
\(\widehat{\sigma} = s \oslash (1 \ominus s)\) requires two floating
point operations, so
\(\widehat{\sigma} = \sigma\left(1 + \theta_2\right)\). This round-off
factor contributes to \(2n\) of the \(5n\) (or \(6n\)) in the coefficient
of \(\widetilde{p}(s)\).

When \(s < 1/2\), we apply Horner's method to the scaled Bernstein
coefficients, however when \(s \geq 1/2\) we reverse the order
before performing Horner's method. As a result, we refer to
\(c_k\) and \(\widetilde{c}_k\) instead of \(b_k\) and
\(\widetilde{b}_k\). In the \(s < 1/2\) case, \(c_k = b_k\) and
in the other \(c_k = b_{n - k}\).
When computing \(p(s)\), we start with \(\widehat{p}_n = c_n\)
and then for \(k = n - 1, \ldots, 0\):
\begin{equation}
\widehat{p}_k = \left[\widehat{\sigma} \otimes \widehat{p}_{k + 1}\right]
  \oplus \left[\binom{n}{k} \otimes c_k\right].
\end{equation}
This means that
\begin{equation}
\widehat{p}_k = \sigma \widehat{p}_{k + 1} (1 + \theta_4)
+ \widetilde{c}_k (1 + \theta_2)
\end{equation}
so that\footnote{We could ignore round-off from the multiplication by
\(\binom{n}{0}\), but we don't.}:
\begin{align}
\widehat{p}_0 &= \sigma \widehat{p}_1 (1 + \theta_4)
+ \widetilde{c}_0 (1 + \theta_2) \\
&= \sigma^2 \widehat{p}_2 (1 + \theta_8) +
  \sigma \widetilde{c}_1 (1 + \theta_6) +
  \widetilde{c}_0 (1 + \theta_2) \\
&\mathrel{\makebox[\widthof{=}]{\vdots}} \nonumber \\
&= \sigma^n \widehat{p}_n (1 + \theta_{4n}) +
  \sum_{j = 0}^{n - 1} \widetilde{c}_j \sigma^j (1 + \theta_{4j + 2}).
\end{align}
Defining \(\widehat{m}_0 = 1, \widehat{m}_1 = m\)
and \(\widehat{m}_{k + 1} = \widehat{m}_k \otimes m\) we'll have
\(\widehat{m}_n = m^n (1 + \theta_{n - 1})\).
Hence, in the final step of the VS algorithm we have
\begin{equation}
\widehat{m}_n \otimes \widehat{p}_0 =
m^n \widehat{p}_0 (1 + \theta_n) = m^n \left[
  \sigma^n \widetilde{c}_n (1 + \theta_{5n}) +
  \sum_{j = 0}^{n - 1} \widetilde{c}_j \sigma^j
  (1 + \theta_{4j + 2 + n})\right].
\end{equation}
When \(s < 1/2\), \(m = \widehat{r} = (1 - s)(1 + \theta_1)\), hence
\(m^n = (1 - s)^n (1 + \theta_n)\), \(\sigma = s/(1 - s)\) and
\begin{align}
\widehat{m}_n \otimes \widehat{p}_0 &= (1 - s)^n \left[
  \sigma^n \widetilde{b}_n (1 + \theta_{6n}) +
  \sum_{j = 0}^{n - 1} \widetilde{b}_j \sigma^j
  (1 + \theta_{4j + 2 + 2n})\right] \\
&= p(s) + (1 - s)^n \left[
  \sigma^n \widetilde{b}_n \theta_{6n} +
  \sum_{j = 0}^{n - 1} \widetilde{b}_j \sigma^j
  \theta_{4j + 2 + 2n}\right] \\
&= p(s) +
  b_n B_{n, n}(s) \theta_{6n} +
  \sum_{j = 0}^{n - 1} b_j B_{n - j, n}(s)
  \theta_{4j + 2 + 2n}
\end{align}
hence
\begin{equation}
\left|p(s) - \mathtt{VS}(p, s)\right| \leq \gamma_{6n}
  \widetilde{p}(s).
\end{equation}

When \(s \geq 1/2\), \(\sigma = (1 - s)/s\). Since \(m = s\), no extra
round-off accumulates in \(m^n\), but the order of the coefficients
is reversed:
\begin{align}
\widehat{m}_n \otimes \widehat{p}_0 &= s^n \left[
  \sigma^n \widetilde{b}_0 (1 + \theta_{5n}) +
  \sum_{j = 0}^{n - 1} \widetilde{b}_{n - j} \sigma^j
  (1 + \theta_{4j + 2 + n})\right] \\
&= p(s) + s^n \left[
  \sigma^n \widetilde{b}_0 \theta_{5n} +
  \sum_{j = 0}^{n - 1} \widetilde{b}_{n - j} \sigma^j
  \theta_{4j + 2 + n}\right] \\
&= p(s) +
  b_0 B_{0, n}(s) \theta_{5n} +
  \sum_{j = 0}^{n - 1} b_{n - j} B_{n - j, n}(s)
  \theta_{4j + 2 + n}
\end{align}
hence
\begin{equation}
\left|p(s) - \mathtt{VS}(p, s)\right| \leq \gamma_{5n}
  \widetilde{p}(s)
\end{equation}
as desired.
\end{proof}

\end{document}